\documentclass{amsart}  

\usepackage{amsrefs}
\usepackage[usenames, dvipsnames]{color}
\usepackage{amsfonts}
\usepackage{amsthm}
\usepackage{amsmath}
\usepackage{amsfonts}
\usepackage{latexsym}
\usepackage{amssymb}
\usepackage{amscd}
\usepackage[latin1]{inputenc}
\usepackage{verbatim}
\usepackage{enumerate}
\usepackage{comment}
\usepackage{mathrsfs}

\newtheorem{theorem}{Theorem}[section]
\newtheorem{lemma}[theorem]{Lemma}
\newtheorem{proposition}[theorem]{Proposition}

\newtheorem{definition}[theorem]{Definition}

\theoremstyle{definition}

\begin{document}

\title[A simplified and unified generalization of some majorization results]{A simplified and unified generalization of some majorization results}

\author[Shirin Moein]{Shirin Moein\textsuperscript{1,2}}

\author[Rajesh Pereira]{Rajesh Pereira\textsuperscript{2}}

\author[Sarah Plosker]{Sarah Plosker\textsuperscript{3,2}}

\thanks{\textsuperscript{1}Department of Mathematical Sciences, Isfahan University of Technology, Isfahan 84156-83111, Iran}

\thanks{\textsuperscript{2}Department of Mathematics and Statistics, University of Guelph, Guelph, ON N1G 2W1, Canada}

\thanks{\textsuperscript{3}Department of Mathematics and Computer Science, Brandon University,
Brandon, MB R7A 6A9, Canada}



\begin{abstract}  
We consider positive, integral-preserving linear operators acting on $L^1$ space, known as stochastic operators or Markov operators. We show that, on finite-dimensional spaces, any stochastic operator can be approximated by a sequence of stochastic integral operators (such operators arise naturally when considering matrix majorization in $L^1$). We collect a number of results for vector-valued functions on $L^1$, simplifying some proofs found in the literature. In particular, matrix majorization and multivariate majorization are related in $\mathbb{R}^n$. In $\mathbb{R}$, these are also equivalent to convex function inequalities. 
\end{abstract}

\maketitle

\section{Introduction}\label{sec:Intro}
 
In this work, we connect  several generalizations of majorization  in reference to vector-valued measurable functions; notably, matrix majorization, multivariate majorization, mixing distance, $f$-divergence, and coarse graining.  
While some results are known, they appear rather obscure in the literature; we also simplify arguments when possible. 

We first recall the definition of (vector) majorization: if $x, y\in \mathbb{R}^n$, we say $x$ is \emph{majorized} by $y$, denoted $x\prec y$, if \begin{eqnarray*}
\sum_{j=1}^{k}x^{\downarrow}_{j}\leq \sum_{j=1}^{k}y^{\downarrow}_{j}\quad \forall k\in \{1,\dots,n-1\}
\end{eqnarray*}
with equality when $k=n$, where $x$ has been reordered so that $x^{\downarrow}_{1}\geq x^{\downarrow}_{2}\geq \cdots \geq x^{\downarrow}_{n}$ (and similarly for $y$). A well-known theorem of Hardy, Littlewood, and P\'{o}lya states that $x\prec y$ is equivalent to the existence of a doubly stochastic matrix $S$ such that $x=Sy$ \cite[Theorem 8]{HLP}.

Consider now two matrices $R\in M_{m\times n}(\mathbb{R})$ and $T\in M_{p\times n}(\mathbb{R})$. We say $R$ is \emph{majorized} by $T$, denoted $R\prec T$ (where it is clear from context that this is \emph{matrix} majorization rather than vector majorization, although matrix majorization  is sometimes denoted $\prec_d$ or $\prec_S$ to distinguish it from vector majorization) if there exists a column stochastic matrix $S\in M_{m\times p}(\mathbb{R})$ such that $R= ST$. For more information on matrix majorization, see \cite{dahl1999}; when we restrict ourselves to the special case where $m=p$ and $S$ is doubly stochastic we get a more restrictive ordering called multivariate majorization, see \cite[Chapter 15]{MOA}. Matrix majorization has recently been generalized to quantum majorization between bipartite states \cite{Gour2017}.

  We denote by $L^1(X, \mu)$, or simply $L^1(X)$ if the measure $\mu$ is clear from context, the set of all functionals $f$ satisfying $\int_X|f|d\mu<\infty$. 
 If $f\in L^1(X)$, the distribution function of $f$   is defined by $d_f(s)= \mu (\{x: f(x)>s\} )$ for all real $s$, 
and the decreasing rearrangement of $f$ is defined by  
\begin{eqnarray*}
f^\downarrow (t)&=& \inf \{s: d_f(s)\leq  t\},\quad \quad\quad 0 \leq  t \leq \mu( X)\\
&=& \sup \{s: d_f(s)> t\}, \quad \quad\quad 0 \leq  t \leq \mu(X).
\end{eqnarray*}

We are now in the position to  define continuous majorization. Typically the word ``continuous'' is dropped as it is clear from context.

\begin{definition}\label{def:cont} Let $(X,  \mu)$ and $(Y,  \nu)$ be finite measure spaces for which $a=\mu( X)=\nu(Y)$.
If $f\in L^1(X,\mu)$ and $g\in L^1( Y, \nu)$ satisfy 
\begin{eqnarray*}
\int_0^t f^\downarrow d x&\leq & \int_0^t g^\downarrow d x\quad \forall t:\, 0\leq  t\leq  a\\
\textnormal{and }\int_0^a f^\downarrow d x&=& \int_0^a g^\downarrow d x, 
\end{eqnarray*}
where the integration is with respect to Lebesgue measure, then we say that  $f$ is majorized by $g$, denoted $f\prec g$.
\end{definition}

Following \cite{dahl1999}, we define the positive homogeneous subadditive functionals on $\mathbb{R}^n$, also called \emph{sublinear} functionals, to be all functionals $\psi$ satisfying $\psi(\lambda x)=\lambda\psi(x)$ and $\psi(x+y)\leq \psi(x)+\psi(y)\ $ for all $x, y\in \mathbb{R}^n$, and $\lambda\geq 0$.


Part of  \cite[Theorem 3.3]{dahl1999} shows that if $R\in M_{m\times n}(\mathbb{R})$ and $T\in M_{p\times n}(\mathbb{R})$, then $R\prec T$ is equivalent to $\sum_{j=1}^m\psi(r_j)\leq \sum_{j=1}^p\psi(t_j)$  for all sublinear functionals $\psi$, where $r_j$ is the $j$th row of the matrix $R$, and similarly for $t_j$.

Given a measure space $(X,\mu)$, let $L^1(X, \mu, \mathbb{R}^n)$, or simply $L^1(X, \mathbb{R}^n)$, denote the set of all measurable functions $f$ from   $(X,\mu)$ to $\mathbb{R}^n$ that satisfy  $\int_X \vert f\vert d\mu <\infty$, where $\vert f \vert(x)=\sum_{k=1}^n |f_k(x)|$. 

The notion of a stochastic matrix was generalized to a  stochastic operator on $L^1([0,1])$ in \cite{RSS1980}; we provide the corresponding definition for a stochastic operator from $L^1(Y, \nu)$ to  $L^1(X, \mu)$. Such an operator is sometimes referred to a   \emph{Markov operator} in the literature \cite{LMbook}.

\begin{definition}
Let $(X,\mu)$ and $(Y,\nu)$ be $\sigma$-finite measure spaces. A linear operator $S:L^1(Y)\to L^1(X)$ is called a {\em stochastic operator} if 
\begin{enumerate}
    \item $S$ is positive (that is, $S$ takes positive elements to positive elements), and
    \item $\int_X Sf d\mu = \int_Y f d\nu, \quad \forall f\in L^1(Y)$.
\end{enumerate}
Moreover, if in addition to the two conditions above, $\mu(X)=\nu(Y)<\infty$ and $S1=1$, then $S$ is called a {\em doubly stochastic operator}.
\end{definition} 

We have the following lemma which will be useful later on.

\begin{lemma} \label{absineq} Let $(X,\mu )$ and $(Y,\nu)$ be $\sigma$-finite measure spaces.  Let $f\in L^1(Y)$ and $S:L^{1}(Y)\to L^{1}(X)$ be a stochastic operator, then $\int_{X}\vert (Sf)(t)\vert d\mu(t)\le \int_{Y}\vert f(y)\vert d\nu(y)$. \end{lemma}

\begin{proof} Let $f_+(y)=\max(f(y),0)$ and $f_-(y)=\max(-f(y),0)$.  Then \begin{eqnarray*} \int_{X}\vert Sf(x) \vert d\mu(x) &\le &\int_{X} S(f_+)(x)d\mu(x) +  \int_{X} S(f_-)(x)d\mu(x)\\ &=& \int_{Y}f_+(y)d\nu(y)+\int_{Y}f_-(y)d\nu(y)\\   & =& \int_{Y}\vert f(y)\vert d\nu(y).\end{eqnarray*} \end{proof}

Note that the absolute value function is a nonnegative sublinear functional on $\mathbb{R}$; we will later show that a similar inequality holds for all nonnegative sublinear functionals on $\mathbb{R}^n$.  We first need to describe how $S$ acts on an element of $ L^1(Y,\mathbb{R}^n)$.
If $f=(f_1,f_2,...,f_n)\in L^1(Y,\mathbb{R}^n)$, then $S$ acts componentwise on $f$; that is,  $Sf=(Sf_1,Sf_2,...,Sf_n)$.

\begin{definition} \label{exp}  Let $(X,\mu)$ be a $\sigma$-finite measure space and let $\mathcal{P}=\{E_i\}_{i\in \mathbb{N}}$ be a partition of $X$ into disjoint measurable sets of finite measure.  We define $M_{\mathcal{P}}$ to be the operator which maps every $f\in L^1(X)$ to $\sum_{i\in \mathbb{N}} a_i \chi_{E_i}$ where $a_i=\frac{1}{\mu(E_i)}\int_{E_i}f(x)d\mu(x)$ if $E_i$ has positive measure and $a_i=0$ if $E_i$ is measure zero. 
\end{definition}

 We note that it is easy to verify that $M_{\mathcal{P}}$ in Definition \ref{exp} is a  stochastic operator on $L^1(X)$; in fact, it maps $1\mapsto 1$ and is therefore a doubly stochastic operator.

Let $K$ be a convex set. A function $f:K\rightarrow K$ is \emph{affine} if $f(\lambda x +(1-\lambda)y)=\lambda f(x)+(1-\lambda)f(y)$ for all $x, y\in K$ and all $\lambda\in (0,1)$. We note that affine functions on the nonnegative face of the unit ball of $L^1$ are exactly the stochastic operators. Affine transformations on measure spaces are used to define coarse graining, a relation on the measurement statistics coming from two positive operator valued measures \cite{BQ1990, BQ1993, Heinonen, HZbook, Zan}. 
A stochastic operator is an affine 
 transformation between nonnegative faces of the unit balls in the respective measure spaces. Note that the $L^1$ norm is one of the few norms where the nonnegative elements of the unit ball form a face.

The following theorem is a combination of two well-known results in the literature.
\begin{theorem}\label{HLP}
If $f\in L^1(X,\mu)$, $g\in L^1(Y,\nu)$ where $\mu(X)=\nu(Y)< \infty$, then the following are equivalent:
\begin{enumerate}
    \item 
$f\prec g$, as in Definition \ref{def:cont}. 
\item For all convex functions $\phi:\mathbb{R}\to \mathbb{R}$, 
$$\int_X \phi(f)d\mu\leq \int_Y \phi(g)d\nu.$$
\item 
There exists a doubly stochastic operator $D:L^1(Y)\to L^1(X)$ such that $f=Dg$.
\end{enumerate}
\end{theorem}
\begin{proof}
Chong \cite{Chong} in Theorem 2.5 proved the equivalence of 1 and 2 and Day \cite{Day} in Theorem 4.9 proved the equivalence of 2 and 3. 
\end{proof}

Of particular interest are the integral operators which are stochastic or doubly stochastic.

\begin{definition}
A \emph{stochastic kernel} is a measurable function $S:X\times Y\to[0,\infty)$ such that $\int_X S(x,y)d\mu(x)=1$ for almost all $y\in Y$.
A \emph{doubly stochastic kernel} is a stochastic kernel with the additional property that $\int_Y S(x,y)d\nu(y)=1$ for almost all $x\in X$.
\end{definition}

\begin{definition} An integral operator $M$ from $L^1(Y)$ to $L^1(X)$ given by $Mg=\int_{Y}S(x,y)g(y)d\nu(y)$ is said to be a \emph{stochastic integral operator} (resp.\ \emph{doubly stochastic integral operator}) if $S(x,y)$ is stochastic kernel (resp.\ doubly stochastic kernel).
\end{definition}

All stochastic integral operators are  stochastic operators, and all doubly stochastic integral operators are doubly stochastic operators. However, the converse of either statement is false. Indeed, consider the identity operator which is a doubly stochastic operator but is not a doubly stochastic integral operator nor a stochastic integral operator.

\section{Convex function inequalities}
We now discuss some properties of convex and sublinear functionals. 

%

\begin{proposition}  \label{prop:sublinHLP} Let $K$ be a convex cone  of $\mathbb{R}^n$.  Let $\phi:K \to \mathbb{R}$ be a convex functional.  Then $\psi (v,x):=x\phi(\frac{v}{x})$ is a sublinear functional on the cone $K\times (0,\infty)$.
\end{proposition}
   
\begin{proof}
For $\lambda>0$, we have $\psi(\lambda(v,x))=\lambda x \phi (\frac{\lambda v}{\lambda x})=\lambda x \phi (\frac{v}{x})=\lambda \psi(v,x)$. 
Next, let $v_1,v_2\in K$ and  $x_1,x_2\in (0,\infty)$. Then by convexity of $\phi$ we have 
$$\phi\Big(\frac{v_1+v_2}{x_1+x_2}\Big)=\phi\Big(\frac{x_1v_1 }{x_1(x_1+x_2)}+ \frac{x_2v_2}{x_2(x_1+x_2)}\Big)\leq \frac{x_1}{x_1+x_2}\phi\Big(\frac{v_1}{x_1}\Big)+\frac{x_2}{x_1+x_2}\phi\Big(\frac{v_2}{x_2}\Big).$$
It follows that $\psi$ is a sublinear functional on $K\times  (0,\infty)$.
\end{proof}

Due to issues with convergence, we have avoided considering $x=0$ in the above proposition, whence the restriction to $ (0,\infty)$. However, one can consider $x\rightarrow 0^+$ to obtain the recession function $\phi_\infty$. Note that the converse of the above proposition is immediate; any sublinear functional on a convex set is automatically convex on that set. 

\begin{proposition}\label{thm:Lipseq}
 Let $K$ be a convex subset of $\mathbb{R}^n$. For any continuous convex  nonnegative functional $\phi$ on $K$, there exists an increasing sequence of Lipschitz convex nonnegative functionals $\{\phi_k\}_{k=1}^\infty$ that converges pointwise to it on $K$. If further, $K$ is a convex cone and $\phi$ is sublinear, then $\{\phi_k\}_{k=1}^\infty$ can be taken to be sublinear. 
\end{proposition} 


If $K$ is a closed convex cone in $\mathbb{R}^n$, then in fact every continuous sublinear functional $\phi:K\to \mathbb{R}$ 
is Lipschitz.

To prove Theorem \ref{contmaj}, we require the generalization of Jensen's inequality to the multivariate case; see \cite[Proposition 16.C.1]{MOA}. 

\begin{theorem}\label{thm:Jensen}(multivariate Jensen's inequality)   Let $(X, \mu)$ be a probability measure space. Let $\phi:\mathbb{R}^n\to \mathbb{R}$ be a convex function and let $f\in L^1(X, \mathbb{R}^n)$. Then   $\phi(\int_{X}f(x) d\mu(x) )\le \int_{X}\phi(f(x)) d\mu(x)$.
\end{theorem}




If $\phi$ is a sublinear function, we no longer require the measure to be a probability measure. 

\begin{theorem}[Roselli-Willem inequality]\label{RW}\cite[Theorem 6]{RW} Let $(X, \mu)$ be an arbitrary measure space. Let $\phi:\mathbb{R}^n\to \mathbb{R}$ be a sublinear function and let $f\in L^1(X, \mathbb{R}^n)$.   Then   $\phi(\int_{X}f(x) d\mu(x) )\le \int_{X}\phi(f(x)) d\mu(x)$.
\end{theorem}



The related concept of $f$-divergence (which we call $\phi$-divergence since $\phi$ is convex) was introduced by Csisz{\'a}r \cite{Csiszar} and was studied extensively by statisticians \cite[Chapter 2]{CohenBook} and \cite{CD05,  JC15, LV06, MM13, SV}.  Similar concepts with different names have also appeared in the Physics literature  \cite[Chapter 6]{QIStext} and \cite{Morimoto, Zan}.

\begin{definition} Let $(X,\mu)$ be a measure space and $V$ be a real vector space.  Let $\phi$ be a real valued convex function on $V$.  Let $f:X \to V$ and let $h:X\to (0,\infty)$ with all functions being measurable.  Then the $\phi$-divergence of $f$ with respect to $h$ is $\int_X h\phi (\frac{f}{h}) \,d\mu$.
\end{definition}
 
The following result is useful in relating $\phi$-divergence inequalities with sublinear functional integral inequalities:

\begin{theorem} \label{thm:all} Let $(X,\mu)$ and $(Y,\nu)$ be measure spaces, $f\in L^1(X,K)$, $g\in L^1(Y,K)$, $h\in  L^1(X,(0,\infty))$ and  $k\in L^1(Y,(0,\infty))$, where $K$ is a convex cone. Define  $F:X \to K \times \mathbb{R}$ and $G:Y \to K \times \mathbb{R}$  as $F(x)=(f(x),h(x))$ and $G(y)=(g(y),k(y))$.      The following are equivalent. 
\begin{enumerate}
\item\label{i:fdiv} The $\phi$-divergence of $f$ with respect to $h$ is less than or equal to that of $g$ with respect to $k$; that is,  $\int_X \phi (\frac{f}{h})h d\mu\leq  \int_Y \phi (\frac{g}{k})k d\nu$, for all convex functions $\phi:K\to \mathbb{R}$. 
\item\label{i:sublin} $\int_X \psi (F(x))d\mu(x)\leq \int_Y \psi (G(y))d\nu(y)$ for all sublinear functionals $\psi:K\times (0,\infty)\to\mathbb{R}$.
\end{enumerate}  \end{theorem}

\begin{proof}

 (\ref{i:fdiv})$\Rightarrow $ (\ref{i:sublin}):
 
 Let $\psi(v,t) :K \times (0,\infty) \to \mathbb{R}$ be a sublinear functional. Let $\phi(v)=\psi(v,1)$ for all $v\in K$.  Then $\phi$ is  a real valued convex function on $K$. We can then use the sublinearity of $\psi$ to prove the following:
 
$\begin{array}{rcl}
\int_X \psi (F(x))d\mu(x)
                         & = & \int_X \phi (\frac{f(x)}{h(x)})h(x)d\mu(x)\\
                         & \le & \int_Y \phi (\frac{g(x)}{k(x)})k(x)d\nu(x)\\
                         & = & \int_Y \psi (G(x))d\nu(x).
 \end{array}$ 
 
  (\ref{i:sublin})$\Rightarrow $ (\ref{i:fdiv}): This implication is a straightforward application of Proposition \ref{prop:sublinHLP}. Let $\phi $ be any real-valued convex functional on $K$.  Then $\psi (v,x)=x\phi(\frac{v}{x})$ is a sublinear functional on $K\times (0,\infty)$, and we have
 
  $\begin{array}{rcl}
  
  \int_X \phi (\frac{f(x)}{h(x)})h(x)d\mu(x) & = &   \int_X \psi(f(x),h(x))d\mu(x)\\
                                              & \le & \int_Y \psi (g(y),k(y))d\nu(y)\\
                                               & = & \int_Y \phi (\frac{g(y)}{k(y)})k(y)d\nu(y).
   \end{array}$ 
   \end{proof}

\section{A generalization of matrix majorization}

With the notion of a stochastic operator from $L^1(Y, \nu)$ to  $L^1(X, \mu)$, we can generalize the definition of matrix majorization:

\begin{definition}\label{defn:matrixmaj}
Let $(X,\mu)$ and $(Y,\nu)$ be measure spaces.  Let $f=(f_1, f_2, \dots, f_n)\in L^1(X,\mathbb{R}^n)$ and $g=(g_1, g_2, \dots, g_n)\in L^1(Y,\mathbb{R}^n)$.  Then we say that $f$ is     \emph{matrix majorized} by $g$, denoted $f\prec_M g$, if there exists a stochastic operator $S$ such that $f=S(g)$; i.e., $f_k=Sg_k$ for all $k=1, \dots, n$. 
\end{definition}

It is straightforward to check that  matrix majorization between measurable functions in $L^1$ is a reflexive, transitive relation and therefore is a preorder, which generalizes the same result in  \cite[Theorem 3.3]{dahl1999} for matrix majorization on matrices.

The term matrix majorization was coined by Dahl \cite{dahl1999}.  To see that our formulation is a generalization of Dahl's, we now restrict ourselves to the special case where $X=\{ 1,2,...,m\}$ and $Y=\{ 1,2,...,p\}$ are finite sets, $\mu$ and $\nu$ are counting measures.  We can represent each function $f=(f_1,f_2,...,f_n)\in L^1(X,\mathbb{R}^n)$ as an $n$ by $m$ 
matrix $A_{f}$ whose $k$th column is $f(k)$.   We can see that $f$ is matrix majorized by $g$ if there exists a row stochastic matrix $S$ such that $A_{f}=A_{g}S$; the later is Dahl's formulation of matrix majorization on matrices.


We now note that part of \cite[Theorem 3.3]{dahl1999} can now be rephrased as follows:

\begin{theorem} Let $X=\{ 1,2,...,m\}$, $Y=\{ 1,2,...,p\}$, $f\in L^1(X,\mathbb{R}^n)$ and $g\in L^1(Y,\mathbb{R}^n)$. Then $f$ is matrix  majorized by $g$ if and only if $\sum_{k=1}^m \phi (f(k)) \le \sum_{k=1}^p \phi (g(k))$ for all sublinear functionals  $\phi: \mathbb{R}^n\to \mathbb{R} $.    \end{theorem}

This suggests the following one-sided extension to the general case. 

\begin{theorem} \label{contmaj} Let $(X,\mu)$ and $(Y,\nu)$ be measure spaces, $f\in L^1(X,\mathbb{R}^n)$ and $g\in L^1(Y,\mathbb{R}^n)$. If there exists a stochastic kernel $S(x,y): X\times Y\rightarrow [0,\infty)$ such that $f(x)=\int_Y S(x,y)g(y)d\nu(y)$ then $\int_X \phi (f(x))d\mu(x) \le \int_Y \phi (g(y))d\nu(y)$ for all sublinear functionals $\phi: \mathbb{R}^n\to \mathbb{R}$.    \end{theorem}


\begin{proof}
 Suppose  there exists a stochastic kernel $S(x,y): X\times Y\rightarrow [0,\infty)$ such that $f(x)=\int_Y S(x,y)g(y)d\mu(y)$.  Let $\phi :\mathbb{R}^n\to \mathbb{R}$ be sublinear. 
 Hence by using Theorem \ref{RW} and Fubini's Theorem

$\begin{array}{rcl}
\int_X \phi (f(x))d\mu(x)& = &\int_X \phi (\int_Y S(x,y)g(y)d\nu(y))d\mu(x) \\ 
   &\le & \int_X \int_Y \phi (S(x,y)g(y))d\nu(y)d\mu(x)\\ 
   & = &  \int_X \int_Y S(x,y)\phi(g(y))d\nu(y)d\mu(x)\\ 
   & = & \int_Y (\int_X S(x,y)d\mu(x))\phi(g(y))d\nu(y)\\ 
   & = & \int_Y \phi(g(y))d\nu(y),
   \end{array}$
   
as desired.
\end{proof}

Note that, if we take $X=Y=[0,1]$ in Theorem \ref{contmaj}, then this is nearly the definition of mixing distance \cite[Definition 1a]{RSS1978}, except that the authors of \cite{RSS1978} take $\phi$ to be any convex functions, or certain subsets thereof, whereas we are working with sublinear (positively homogeneous convex) functionals in accordance with \cite{dahl1999}. This similarity hints at a connection between matrix majorization and the mixing distance. 

\begin{lemma}\label{seqpart}  Let $(X,\mu)$ be a $\sigma$-finite measure space and let $f\in L^1(X)$.  Then there exists a sequence of partitions $\{\mathcal{P}_n\}_{n=1}^\infty$ of $X$ into disjoint sets of finite measure such that $\{M_{\mathcal{P}_n}f\}_{n=1}^\infty$ converges to $f$ in the $L^1$ norm. \end{lemma}

\begin{proof}  Let $P_n=\{E_k \}_{k=0}^{2n^2+1}$ where $E_0=\{ x\in X: f(x)<-n\}$  and $E_j=\{x\in X: f(x)\in [-n+\frac{j-1}{n},-n+\frac{j}{n})\} $ if $1\le j\le 2n^2$ and $E_{2n^2+1}=\{ x\in X: f(x)\ge n\}$.  If $E_k$ has infinite measure for some $k\in \{0, \dots, 2n^2+1\}$, since $X$ is $\sigma$-finite, $E_k$ is a countable disjoint union of sets of finite measure. Replace every such $E_k$ in the partition with the sets of finite measure.  It is then easy to verify that $\{M_{\mathcal{P}_n}f\}_{n=1}^\infty$ converges to $f$ in the $L^1$ norm. \end{proof}

We note that if  $\mathcal{P}=\{E_i\}_{i\in \mathbb{N}}$ and  $\mathcal{Q}=\{F_i\}_{j\in \mathbb{N}}$ are two partitions of $X$ into disjoint sets of finite measure, then we can form  the intersection partition $\mathcal{P}\cap \mathcal{Q}=\{E_i\cap F_j: i,j\in \mathbb{N}\}$. 

\begin{lemma}\label{seqpart2}  Let $(X,\mu)$ be a $\sigma$-finite measure space and let $V$ be a finite dimensional subspace of $L^1(X)$.  Then there exists a sequence of partitions $\{\mathcal{P}_n\}_{n=1}^\infty$ of $X$  into disjoint sets of finite measure such that $\{M_{\mathcal{P}_n}f\}_{n=1}^\infty$ converges to $f$ in the $L^1$ norm for all $f\in V$. \end{lemma}

\begin{proof} The proof is by induction on the dimension of $V$ with the base case being Lemma \ref{seqpart}.  Now suppose the induction hypothesis holds for dimension $n$ and let $V$ be a subspace of dimension $n+1$.  Let $S$ be a subspace of $V$ of dimension $n$; by the induction hypothesis there exists a sequence of partitions $\{\mathcal{P}_n\}_{n=1}^\infty$ of $X$ into disjoint sets of finite measure such that $\{M_{\mathcal{P}_n}f\}_{n=1}^\infty$ converges to $f$ in the $L^1$ norm for all $f\in S$.  Now let $g\in V$ with $g\not \in S$, then by Lemma \ref{seqpart} there exists a sequence of partitions $\{\mathcal{Q}_n\}_{n=1}^\infty$ into disjoint sets of finite measure such that $\{M_{\mathcal{Q}_n}g\}_{n=1}^\infty$ converges to $g$ in the $L^1$ norm.  It is easy to see that sequence of intersection partitions  $\{\mathcal{R}_n=\mathcal{P}_n\cap \mathcal{Q}_n\}_{n=1}^\infty$ now satisfies the property that that $\{M_{\mathcal{R}_n}f\}_{n=1}^\infty$ converges to $f$ in the $L^1$ norm for all $f\in V$.  \end{proof}

\begin{theorem} \label{sto:markov} Let $(X,\mu )$ and $(Y,\nu)$ be $\sigma$-finite measure spaces.  Let $S:L^{1}(Y)\to L^{1}(X)$ be a stochastic operator and let $V$ be a finite dimensional subspace of $L^{1}(Y)$.  Then there exists a sequence of stochastic integral operators from $L^{1}(Y)$ to $L^{1}(X)$ which converges to $S$ on $V$. \end{theorem}

\begin{proof} Let $\mathcal{P}=\{E_i\}_{i\in \mathbb{N}}$ be a partition of $X$ into disjoint sets of finite measure.  Then $M_{\mathcal{P}}S$ is a stochastic operator from $L^{1}(Y)\to L^{1}(X)$; we will show that it is a stochastic integral operator. Fix $x\in X$. Then there exists a unique $k$ such that $x\in E_k$.  Define the functional $g_x(f)=(M_{\mathcal{P}}Sf)(x)$.  Then $$\begin{array}{rcl}  \vert g_x(f)\vert &=& \vert \frac{1}{\mu(E_k)}\int_{E_k}(Sf(t))d\mu(t)\vert \\  
& \le&  \frac{1}{\mu(E_k)}\int_{X}\vert (Sf)(t)\vert d\mu(t) \\ & \le & \frac{1}{\mu(E_k)}\int_{Y}\vert f(y)\vert d\nu(y)\quad \textnormal{by Lemma \ref{absineq} }\\ & = &\frac{1}{\mu(E_k)}\Vert f\Vert_1.\end{array}$$  Hence $g_x$ is a bounded linear functional of $L^1(Y)$.  So by the Riesz representation theorem, there exist a nonnegative function $h_x\in L^{\infty}(Y)$ such that $g_x(f)=\int_Y f(y)h_x(y)d\nu(y)$.  Now let $K_{\mathcal{P}}(x,y)=h_x(y)$ for all $x\in X$ and all $y\in Y$. Since $K_{\mathcal{P}}(x,y)=\sum_{i\in \mathbb{N}}\chi_{E_i}(x)h_x(y)$,  $K_{\mathcal{P}}(x,y)$ is measurable. Then $(M_{\mathcal{P}}Sf)(x)=\int_Y K_{\mathcal{P}}(x,y)f(y)d\nu(y)$.  

Since $M_{\mathcal{P}}S$ is a stochastic operator, we have 
$$\int_X ((M_{\mathcal{P}}S)f)(x) d\mu(x)=\int_Y f(y)d\nu(y)\quad \forall f\in L^1(Y)$$ and by using Fubini's Theorem, we find
$$\begin{array}{rcl}
\int_Y f(y)d\nu(y) &=& \int_X ((M_{\mathcal{P}}S)f)(x) d\mu(x)\\ &=& \int_X\int_Y K_{\mathcal{P}}(x,y)f(y)d\nu(y)d\mu(x)\\
&=&\int_Y f(y) (\int_X K_{\mathcal{P}}(x,y)d\mu(x))d\nu(y).
\end{array}$$

Therefore $\int_X K_{\mathcal{P}}(x,y)d\mu(x)=1$ for almost all $y\in Y$.
Since $V$ is a finite dimensional subspace of $L^{1}(Y)$, the forward image $S(V)$ is a finite dimensional subspace of $L^{1}(X)$.  The result now follows from Lemma \ref{seqpart2}.
\end{proof}
We also have a doubly stochastic version of this theorem:
\begin{theorem}
Let $(X,\mu )$ and $(Y,\nu)$ be finite measure spaces. A doubly stochastic operator $D:L^1(Y)\to L^1(X)$ on a finite dimensional subspace $V$ of $L^1(Y)$ can be approximated by doubly stochastic integral operators.
\end{theorem}
\begin{proof}
 Let $\mathcal{P}=\{E_i\}_{i=1}^n$ be a partition of $X$ into disjoint sets of finite measure. 
The  operator $M_\mathcal{P}:L^1(X)\to L^1(X)$ from Definition \ref{exp} is a doubly stochastic operator and since the composition of doubly stochastic operators is a doubly stochastic operator, $M_\mathcal{P}D$ is a doubly stochastic operator. By a proof similar to that of Theorem \ref{sto:markov}, for all $f\in L^1(Y)$ we have  $(M_{\mathcal{P}}Df)(x)=\int_Y K_{\mathcal{P}}(x,y)f(y)d\nu(y)$ such that  $\int_X K_{\mathcal{P}}(x,y)d\mu(x)=1$ for almost all $y\in Y$. Now suppose $f=1$. Then  $\int_Y K_{\mathcal{P}}(x,y)d\nu(y)=1$ for almost all $x\in X$ and hence $M_{\mathcal{P}}D$ is a doubly stochastic integral operator.
\end{proof}

   \begin{theorem}\label{main2}
   Let $(X,\mu)$ and $(Y,\nu)$ be two $\sigma$-finite measure spaces, $f\in L^1(X,\mathbb{R}^n)$ and $g\in L^1(Y,\mathbb{R}^n)$. If 
  $f$ is matrix  majorized by $g$, 
  then 
   $$\int_X \phi (f(x))d\mu(x)\leq \int_Y \phi (g(y))d\nu(y)$$ for all nonnegative sublinear functionals
  $ \phi: \mathbb{R}^n\rightarrow [0,\infty)$.
   
   \end{theorem}
   \begin{proof}
Let $g=(g_1,...,g_n)$ and   $V=\operatorname{span}\{g_1,..., g_n\}$. Since $V$ is a finite dimensional subspace of $L^1(Y)$, by using Theorem \ref{sto:markov}, there exists a sequence of stochastic integral operators   $\{S_k\}_{k=1}^\infty$ which converges to  the stochastic operator $S$ coming from Definition \ref{defn:matrixmaj}.  
Now by using Theorem \ref{contmaj}, for each $k\in\mathbb{N}$ we obtain $\int_X \phi(S_k g)d\mu(x)\leq \int_Y \phi(g)d\nu(y)$ for all  sublinear functionals $\phi$.

Since $\phi$ is sublinear on  $\mathbb{R}^n$, it is Lipschitz; denote its Lipschitz constant as $c$.  Then we have,
$$\int_{X}|\phi(S_kg)-\phi(Sg)|d\mu(x) \leq c\sum_{j=1}^n\int_X|S_kg_j-Sg_j|d\mu(x)\quad \forall k\in\mathbb{N}, \, \textnormal{ where } c\geq 0.$$
Since $\lim_{k\to \infty} S_kg_j=Sg_j$ in $L^1$ for all $j$, the left hand side must go to zero which means that $\lim_{k\to \infty }\int_{X}\phi(S_kg)d\mu(x)=\int_{X}\phi(Sg)d\mu(x) $.  Therefore we have
$$\int_X \phi(f)d\mu(x)=\int_X \phi(S g)d\mu(x)\leq \int_Y \phi(g)d\nu(y).$$
   \end{proof}
   
  We note that a special case of this result is a slight generalization of a theorem of Alberti; when $X=Y$ and $\mu=\nu$, Theorem \ref{main2} reduces to one direction of the following result which was proved using methods from the theory of von Neumann algebras. 
  
  \begin{theorem}  \cite[Theorem 1]{Alberti}
   Let $(X,\mu)$  be  a $\sigma$-finite measure space and  $f,g\in L^1(X,\mathbb{R}^n)$.  Then
  $f$ is matrix  majorized by $g$, 
  if and only if 
   $$\int_X \phi (f(x))d\mu(x)\leq \int_X \phi (g(x))d\mu(x)$$ for all nonnegative sublinear functionals
  $ \phi: \mathbb{R}^n\rightarrow [0,\infty)$.
 \end{theorem}
  
 We do not know if the converse to Theorem \ref{main2} holds for arbitrary measures.
  
  We now consider a generalization of majorization known as multivariate majorization. In the setting of $\mathbb{R}^n$, we can show (Theorem \ref{thm:DSmultivar}) that  matrix majorization and multivariate majorization are strongly related.
  
  \begin{definition}
  Let $(X,\mu)$ and $(Y,\nu)$ be finite measure spaces, $f\in L^1(X,\mathbb{R}^n)$,  and $g\in L^1(Y,\mathbb{R}^n)$. Then $f$ is \emph{multivariate majorized} by $g$ if there exists a doubly stochastic operator $D:L^1(Y)\to L^1(X)$ such that  $f=Dg$. 
  \end{definition}

 \begin{theorem}\label{thm:DSmultivar}
Let $(X,\mu)$ and $(Y,\nu)$ be finite measure spaces, $f\in L^1(X,\mu ,\mathbb{R}^n)$, $g\in L^1(Y,\nu ,\mathbb{R}^n)$, $h\in  L^1(X,\mu, (0,\infty))$, and  $k\in L^1(Y,\nu, (0,\infty))$. The following are equivalent:
\begin{enumerate}
    \item \label{i1} $(f_1, f_2, \dots, f_n, h)$ is matrix majorized by $(g_1, g_2, \dots, g_n, k)$; i.e.,
    there exists a stochastic operator $S:L^1(Y,\nu)\to L^1(X,\mu)$ such that $Sg_i=f_i$ for all $i=1,...,n$ and $Sk=h$,
    \item \label{i2}$\left(\frac{f_1}{h}, \frac{f_2}{h}, \dots, \frac{f_n}{h}\right)$ is multivariate majorized by $\left(\frac{g_1}{k}, \frac{g_2}{k}, \dots, \frac{g_n}{k}\right)$ with respect to measures $\alpha$ and $\beta$ where the measures $\alpha$ and $\beta$ are defined by 
    $\alpha=h\, d\mu$ and $\beta=k\, d\nu$; i.e.,
    there exists a doubly stochastic operator $D:L^1(Y,\beta)\to L^1(X,\alpha)$ such that $D\frac{g_i}{k}=\frac{f_i}{h}$ for all $i=1,...,n$.
    \end{enumerate}
\end{theorem}

\begin{proof}
 Let  $T_s$ denote the multiplication operator which maps any function $f$ to the product $sf$.
 
 (\ref{i1}) $\Rightarrow (\ref{i2})$: 
 Suppose there exists a stochastic operator $S:L^1(Y,\nu)\rightarrow L^1(X,\mu)$ such that $Sg_i=f_i$ for all $i=1, \dots, n$ and $Sk=h$. We now show that \newline $\left(\frac{f_1}{h}, \frac{f_2}{h}, \dots, \frac{f_n}{h}\right) \in L^1(Y, \beta)$ is multivariate majorized by $\left(\frac{g_1}{k}, \frac{g_2}{k}, \dots, \frac{g_n}{k}\right) \in L^1(X, \alpha)$.

   The multiplication operator $T_{1/h}$ is a stochastic operator from $L^1(X,\mu)$ to $L^1(X,\alpha)$.  Note that for all $i=1, \dots, n$,  $f_i\in L^1(X,\mu)$ and $T_{1/h}(f_i)=\frac{f_i}{h}$. Similarly,  $T_{k}$ is a stochastic operator from $L^1(Y,\beta)$ to $L^1(Y,\nu)$ and   for all $i=1, \dots, n$, $g_i\in L^1(Y,\beta)$ and  $T_{k}(g_i)=kg_i$. Construct $D=T_{1/h}ST_{k}:L^1(Y,\beta)\to L^1(X,\alpha)$, which  is a stochastic operator since it is a product of stochastic operators. Furthermore, $D1=1$,   therefore $D$ is a doubly stochastic operator that maps $\frac{g_i}{k}$ to $\frac{f_i}{h}$. 
   
 (\ref{i2}) $\Rightarrow (\ref{i1})$:   Assume there exists a  doubly stochastic $D:L^1(Y, \beta)\to L^1(X,\alpha)$ such that $\frac{f_i}{h}=D\frac{g_i}{k}$. We define a stochastic operator $T_{h}:L^1(X,\alpha)\to L^1(X,\mu)$ such that for all $i=1, \dots, n$, $f_i\in L^1(X,\mu )$, and $T_{h}(f_i)=f_ih$. Similarly, define $T_{1/k}:L^1(Y,\nu)\to L^1(Y,\beta)$ such that for all  $i=1, \dots, n$, $g_i\in L^1(Y,\nu)$ and $T_{1/k}(g_i)=\frac{g_i}{k}$. Construct  $S=T_{h}DT_{1/k}:L^1(Y,\nu)\to L^1(X,\mu)$. Since $D$ is a (doubly) stochastic operator and the product of stochastic operators is a stochastic operator, $S$ is a stochastic operator such that $Sg_i=f_i$ for all $i=1,...,k$ and $Sk=h$.
\end{proof}

In the setting of $\mathbb{R}$ we can show that  matrix majorization,  multivariate majorization, and the convex function inequalities are all strongly related. The following theorem can be viewed as a simplified version of the result on mixing distance in \cite{RSS1978}.  

 \begin{theorem}
Let $(X,\mu)$ and $(Y,\nu)$ be finite measure spaces, $f\in L^1(X,\mathbb{R})$, $g\in L^1(Y,\mathbb{R})$, $h\in  L^1(X,(0,\infty))$,  $k\in L^1(Y,(0,\infty))$, with $\int_X h \,d\mu=\int_Y k \,d\nu$. The following are equivalent:
\begin{enumerate}
\item \label{i:S} There exists a stochastic operator $S:L^1(Y, \nu)\rightarrow L^1(X, \mu)$ such that $Sg=f$ and $Sk=h$. 
    \item \label{i:C} For all real valued convex functions on $\mathbb{R}$,
$$\int_X \phi \left(\frac{f}{h}\right)h \,d\mu\leq  \int_Y \phi \left(\frac{g}{k}\right)k \,d\nu. $$
\item \label{i:DS}
There exists a  doubly stochastic $D:L^1(Y, \beta)\to L^1(X, \alpha)$ such that $D\left(\frac{g}{k}\right)=\frac{f}{h}$, where the measures $\alpha$ and $\beta$ are defined by
 $\alpha=h\, d\mu$ and $\beta=k\, d\nu$. 
\end{enumerate}
\end{theorem}
\begin{proof}


Both (\ref{i:C}) and (\ref{i:DS}) are equivalent to $\frac{g}{k}\prec \frac{f}{h}$ by Theorem \ref{HLP}.  
%
%
%
%
%
The equivalence of (\ref{i:S}) and (\ref{i:DS}) follows from Theorem \ref{thm:DSmultivar} with $n=1$. 
\end{proof}

\section*{Acknowledgements}
S.M.\  was supported by a travel grant from the Iranian Ministry of Science Research and Technology. She gratefully acknowledges  the University of Guelph and Brandon University for the time she spent at the respective universities during the course of this work. R.P.\ was supported by NSERC Discovery Grant number 400550. S.P.\ was supported by NSERC Discovery Grant number 1174582, the Canada Foundation for Innovation (CFI) grant number 35711, and the Canada Research Chairs (CRC) Program grant number 231250.  The authors thank the anonymous referee for their helpful comments. 
   

\end{document}